\documentclass[11pt]{article}
\usepackage{amsmath,amsfonts,amssymb,amsthm}

\theoremstyle{plain}						%% od rancy
\newtheorem{proposition}{Proposition}[section]
\newtheorem{lemma}[proposition]{Lemma}
\newtheorem{theorem}[proposition]{Theorem}
\newtheorem{corollary}[proposition]{Corollary}
\newtheorem{question}[proposition]{Question}

\theoremstyle{definition}					%% od rancy

\theoremstyle{remark}						%% od rancy

\font\teneufm	= eufm10
\font\seveneufm	= eufm7
\font\fiveeufm	= eufm5

\font\tenscr	= rsfs10 %scaled\magstep1
\font\sevenscr	= rsfs7  %scaled\magstep1
\font\fivescr	= rsfs5  %scaled\magstep1

\font\tenbbm	= bbm10
\font\sevenbbm	= bbm7
\font\fivebbm	= bbm5

\newfam\eufmfam
\textfont\eufmfam=\teneufm
\scriptfont\eufmfam=\seveneufm
\scriptscriptfont\eufmfam=\fiveeufm
\def\eufm{\fam\eufmfam \teneufm}

\newfam\scrfam
\textfont\scrfam=\tenscr
\scriptfont\scrfam=\sevenscr
\scriptscriptfont\scrfam=\fivescr
\def\scr{\fam\scrfam}

\newfam\bbmfam
\textfont\bbmfam=\tenbbm
\scriptfont\bbmfam=\sevenbbm
\scriptscriptfont\bbmfam=\fivebbm

\def\bbb{{\eufm b}} 		% bounding number
\def\ddd{{\eufm d}} 		% dominating number
\def\ccc{{\eufm c}} 		% continuum
\def\D {{\mathcal D}}       % finite subfamily of a filter base
\def\E {{\mathcal E}}       % finite subfamily of a filter base
\def\I {{\scr I}}		% ideal I\def\D {{\mathcal D}}       % finite subfamily of a filter base
\def\J {{\scr J}}		% ideal J
\def\F {{\scr F}}		% filtr F
\def\G {{\mathcal G}}       % filter G
\def\U {{\scr U}}		% ultrafiltr U
\def\V {{\scr V}}		% ultrafiltr V
\def\EE {{\mathbb E}}        % poset
\def\OO {{\mathbb O}}        % poset
\def\PP {{\mathbb P}}        % poset
\begin{document}

\begin{center}
{\LARGE Rapid ultrafilters and summable ideals}\\[4mm]
\textsc{Jana Fla\v{s}kov\'a}
\footnote{Work done patially during a visit to the Institut Mittag-Leffler (Djursholm, Sweden) and partially supported from the European Science Foundation in the realm of the activity entitled 'New Frontiers of Infinity: Mathematical, Philosophical and Computational Prospects'.}\\[3mm]
{\large Department of Mathematics, University of West Bohemia}
\end{center}

\begin{abstract}
This note answers a question raised in \cite{FALC}: Is it consistent that for
an arbitrary tall summable ideal $\I_g$ there exists an $\I_g$-ultrafilter
which is not rapid? We show that assuming Martin's Axiom for $\sigma$-centered posets 
such ultrafilters exist for every tall summable ideal $\I_g$.
\end{abstract}

\section{Introduction}

This note follows up the author's paper ``$\I$-ultrafilters and
summable ideals" \cite{FALC} in which the connections between rapid
ultrafilters and $\I_g$-ultrafilters have been studied. We will use the same
notation and recall the most important definitions and facts in this
introduction.
\smallskip

An ultrafilter $\U$ is called a {\it rapid ultrafilter\/} if the
enumeration functions of sets in $\U$ form a dominating family in
$({}^{\omega}\omega, \leq^{\ast})$, where the enumeration function of a set $A$ is
the unique strictly increasing function $e_A$ from $\omega$ onto $A$.
An ultrafilter $\U$ is called a {\it $Q$-point\/} if for every partition
$\{Q_n: n \in \omega\}$ of $\omega$ into finite sets there is $A \in \U$
such that $|A \cap Q_n| \leq 1$ for every $n \in \omega$.
Clearly, every $Q$-point is a rapid ultrafilter, but the converse is not true (see e.g.~\cite{Mi}).
\smallskip

For a function $g: \omega \rightarrow (0,+\infty)$ such that $\sum\limits_{n \in \omega} g(n) = +\infty$ the family
$$\I_g = \{A \subseteq \omega: \sum_{a\in A} g(a) < +\infty\}$$
is an ideal on $\omega$, which we call the {\it summable ideal determined by function\/} $g$. A summable ideal is tall if and only if $\lim\limits_{n \rightarrow \infty} g(n) = 0$.

The following description of rapid ultrafilters can be found in \cite{V}:
\begin{theorem} \label{MainRapid}
For an ultrafilter $\U \in \omega^{\ast}$ the following are equivalent:
\begin{enumerate}
\item $\U$ is rapid
%\item $\U$ is a weak $\I$-ultrafilter for every tall summable ideal $\I$
%\item $\U$ is an $\I$-friendly ultrafilter for every tall summable ideal $\I$
\item $\U \cap \I_g \neq \emptyset$ for every tall summable ideal $\I_g$
\end{enumerate}
\end{theorem}

The definition of an $\I$-ultrafilter was given by Baumgartner in \cite{B}:
Let $\I$ be a~family of subsets of a~set $X$ such that $\I$ contains all singletons and is
closed under subsets. Given an ultrafilter $\U$ on $\omega$, we say that ${\U}$ is an {\it
$\I$-ultrafilter\/} if for every function $F:\omega \rightarrow X$ there exists $A \in \U$
such that $F[A] \in \I$.

We say that an ultrafilter $\U$ is a {\it hereditarily rapid
ultrafilter\/} if it is a rapid ultrafilter such that for every
$\V \leq_{RK} \U$ the ultrafilter $\V$ is again a rapid
ultrafilter.
Since every hereditarily rapid ultrafilter is obviously a rapid ultrafilter,
the existence of hereditarily rapid ultrafilters is not provable in ZFC
because Miller proved in \cite{Mi} that there are no rapid ultrafilters in Laver model.
On the other hand, every selective ultrafilter is hereditarily
rapid, thus the existence of hereditarily rapid ultrafilters is consistent with
ZFC.

The following characterization of hereditarily rapid ultrafilters follows from
the definition, Theorem \ref{MainRapid} and from the fact that the class of
$\I$-ultrafilters is downwards closed with respect to the Rudin-Keisler
order on ultrafilters.
\begin{theorem}   \label{herrapid}
For an ultrafilter $\U \in \omega^{\ast}$ the following are equivalent:
\begin{enumerate}
\item $\U$ is hereditarily rapid
\item $\U$ is an $\I$-ultrafilter for every tall summable ideal $\I$
\end{enumerate}
\end{theorem}
\medskip

It was proved in \cite{FALC} that $Q$-points (and consequently rapid ultrafilters)
need not be $\I_g$-ultrafilters in a strong sense.

\begin{theorem} \label{QnotI}
$(MA_{\small \rm ctble})$ There is a $Q$-point which is not an $\I_g$-ultrafilter
for any summable ideal $\I_g$.
\end{theorem}

\begin{corollary} \label{rapidnotI}
$(MA_{\rm ctble})$ For an arbitrary summable ideal $\I_g$ there exists a rapid
ultrafilter which is not an $\I_g$-ultrafilter.
\end{corollary}

We also showed in \cite{FALC} that $\I_g$-ultrafilters need not be $Q$-points
by the following counterpart of Theorem \ref{QnotI}.

\begin{theorem} \label{InotQ}
$(MA_{\rm ctble})$ There exists $\U \in \omega^{\ast}$ such that $\U$ is an
$\I_g$-ultrafilter for every tall summable ideal $\I_g$ and $\U$ is not a
$Q$-point.
\end{theorem}

However, we did not prove a counterpart for Corollary \ref{rapidnotI}.
Asssuming Martin`s axiom for countable posets an $\I_{1/n}$-ultrafilter which
is not rapid was constructed in \cite{FTh}, but the question remained open for
an arbitrary tall summable ideal $\I_g$.

The aim of this note is to provide a construction of an $\I_g$-ultrafilter
which is not rapid for an arbitrary tall summable ideal $\I_g$.

\section{Some properties of the summable ideals}

Let us first recall the definition of Kat\v{e}tov order $\leq_K$ for ideals on
$\omega$:
For $\I$ and $\J$ ideals on $\omega$ we write $\I \leq_K \J$ if there is a
function $f: \omega \rightarrow \omega$ such that $f^{-1}[A] \in \J$ for all
$A \in \I$.

The structure of the summable ideals ordered by Kat\v{e}tov order was
investigated by Meza \cite{M-A}. We are particularly interested how the
comparability of two ideals in Kat\v{e}tov order reflects to the inclusion 
of the corresponding classes of $\I$-ultrafilters.

Obviously, if $\I \leq_K \J$ then every $\I$-ultrafilter is a $\J$-ultrafilter.
This implication cannot be reversed in general. However, in Theorem
\ref{IgnotIh} we prove that assuming Martin's Axiom for $\sigma$-centered posets 
the converse is also true whenever $\I$ and $\J$ are tall summable ideals.

\smallskip

From now on all summable ideals will be tall and determined by a decreasing function $g$ 
(notice that every tall summable ideal can be mapped to such an ideal by a permutation). 
These ideals are invariant with respect to the translation which is formulated more precisely 
in the next lemma. For the sake of simplicity of its formulation let us fix the following notation: 
If $A$ is a subset of $\omega$ enumerated increasingly as $A=\{a_n: n \in \omega\}$
then $A+1 = \{a_n+1: n \in \omega\}$.

\begin{lemma} \label{translates}
Assume $\I_g$ is a tall summable ideal determined by a decreasing function $g$, $A$ is a subset of $\omega$ and $B \subseteq A$.
Then
\begin{enumerate}
\item $A \in \I_g$ if and only if $A+1 \in \I_g$
\item $A \in \I_g$ if and only if $B+1 \cup (A\setminus B) \in \I_g$
\end{enumerate}
\end{lemma}

\begin{proof}
1. Since the function $g$ is decreasing, $g(a_n) \geq  g(a_n+1) \geq g(a_{n+1})$.
Thus for every $A \subseteq \omega$ the following inequalities hold
$$\sum_{a\in A} g(a) = \sum_{n \in \omega} g(a_n) \geq \sum_{n \in \omega} g(a_n+1) = \sum_{a\in A+1} g(a)$$
and  
$$\sum_{a\in A} g(a)= \sum_{n \in \omega} g(a_n) \leq g(0) + \sum_{n \in \omega} g(a_n+1)=g(0) + \sum_{a\in A+1} g(a)$$
which implies that $A \in \I_g$ if and only if $A+1 \in \I_g$.

2. follows directly from 1.: $A \in \I_g$ if and only if both $B \in \I_g$ and $A \setminus B \in \I_g$. This is by 1. equivalent to $B+1 \in \I_g$ and $A \setminus B \in \I_g$ which holds if and only if $B+1 \cup A \setminus B \in \I_g$.
\end{proof}

\begin{lemma} \label{choosingA}
Assume $f \in \omega^{\omega}$, $\I_g$ and $\I_h$ are tall summable ideals with
$\I_g \not\leq_K \I_h$. If $H$ is an infinite subset of $\omega$ such that $H
\not\in \I_h$ and $f[H] \not\in \I_g$ then there exists $A \subseteq f[H]$ such
that $A \in \I_g$ and $f^{-1}[A] \cap H \not\in \I_h$.
\end{lemma}

\begin{proof}
%We may assume that both $g$ and $h$ are monotone functions because every tall
%summable ideal is isomorphic to one determined by monotone function.
Let us denote by $\EE$ the set of all even numbers and $\OO$ the set of all odd numbers.

Define $\tilde{f}: \omega \rightarrow \omega$ by
$$\tilde{f}(n) = \left\{\begin{array}{ll}
f(n) & \hbox{if } n \in H \cap f^{-1}[\EE] \hbox{ or } n \in (\omega \setminus H) \cap f^{-1}[\OO] \\
f(n)+1 & \hbox{if } n \in H \cap f^{-1}[\OO] \hbox{ or } n \in (\omega \setminus H) \cap f^{-1}[\EE].
\end{array} \right.$$
Notice that the sets $\tilde{f}[H]$ and $\tilde{f}[\omega \setminus H]$ are disjoint because 
$\tilde{f}[H] \subseteq \EE$ and $\tilde{f}[\omega \setminus H] \subseteq \OO$.
Since $f[H] \not\in \I_g$ and $\tilde{f}[H] = (f[H] \cap \EE) \cup (f[H] \setminus \EE)+1$, 
we have $\tilde{f}[H] \not\in \I_g$ by Lemma \ref{translates}.

It follows from $\I_g \not\leq_K \I_h$ that $\I_g \restriction \tilde{f}[H] \not\leq_K \I_h$, 
so there exists a set $\tilde{A} \in \I_g \restriction \tilde{f}[H]$ such that
$\tilde{f}^{-1}[\tilde{A}] \not\in \I_h$. Put $A = f[\tilde{f}^{-1}[\tilde{A}]]$. 
It remains to verify that $A$ has all the required properties:

$\bullet$ $A \subseteq f[H]$ because $\tilde{f}^{-1}[\tilde{A}] \subseteq H$. 

$\bullet$ $A \in \I_g$ by Lemma \ref{translates} because $\tilde{A} \in \I_g$ and $\tilde{A} = (A \cap \EE) \cup (A \setminus \EE)+1$

$\bullet$ $f^{-1}[A] \cap H \supseteq \tilde{f}^{-1}[\tilde{A}] \cap H = \tilde{f}^{-1}[\tilde{A}]
\not\in \I_h$.
\end{proof}

\begin{lemma} \label{oneset}
($MA_{\sigma-\hbox{\small\rm centered}}$) Assume $\I_h$ is a tall summable ideal and $\F$ is a filter base with $|\F| < \ccc$ 
such that $\F \cap \I_h = \emptyset$.
Then there exists a set $H \subseteq \omega$ such that $H \not\in \I_h$ and $H \setminus F$ is finite for every $F \in \F$.
\end{lemma}

\begin{proof}
%Enumerate $\F = \{F_{\alpha}: \alpha < \kappa\}$.
Define a poset $$\PP = \{\langle K, \D \rangle : K \in [\omega]^{<\omega}, \D \in [\F]^{<\omega}\}$$ with partial order given by $\langle K, \D \rangle \leq_{\PP} \langle L, \E \rangle$ iff $K \supseteq L$, $\min K \setminus L > \max L$, $K \setminus L \subseteq \bigcap \E$ and $\D \supseteq \E$.
It is not difficult to see that $(\PP, \leq_{\PP})$ is a $\sigma$-centered poset.

Now for every $m \in \omega$ define $B_m = \{\langle K, \D \rangle \in \PP: \sum_{k \in K} h(k) \geq m\}$ and for every $F \in \F$ put $B_F =\{\langle K, \D \rangle \in \PP: F \in \D\}$.

\medskip
\noindent
{\it Claim. $B_m$ and $B_F$ are dense in $\PP$ for every $m \in \omega$ and for every $F \in \F$.}

\noindent
Consider arbitrary $\langle L, \E \rangle \in \PP$. Since $\bigcap \E \not \in \I_h$ and $\sum_{k \in \bigcap \E} h(k) = +\infty$, there exists $L' \subseteq \bigcap \E$ with $\min L' > \max L$ such that $\sum_{k \in L'} h(k) \geq m$. Put $K = L \cup L'$ and notice that $\langle K, \E \rangle \leq_{\PP} \langle L, \E \rangle$ and $\langle K, \E \rangle \in B_m$. For the second part put $\D= \E \cup \{F\}$ and observe that $\langle L, \D \rangle \leq_{\PP} \langle L, \E \rangle$ and $\langle L, \D \rangle \in B_F$. \hfill $\Box$
\medskip

According to the assumption $MA_{\sigma-\hbox{\small centered}}$ there exists a generic filter $\G$ on $\PP$.
Define $G = \bigcup \{K \in [\omega]^{<\omega}: (\exists \D \in [\F]^{<\omega}) \langle K, \D \rangle \in \G\}$.
\smallskip

(1) $G \not\in \I_h$

\noindent
For every $m \in \omega$ and every $K \in \G \cap B_{m}$ we have $G \supset K$ 
and $\sum_{k \in K} h(k) \geq m$. Thus $\sum_{k \in G} h(k) = +\infty$ and $G \not\in \I_h$.
\smallskip

(2) $(\forall F \in \F)$ $G \subseteq^{\ast} F$

\noindent
For every $F \in \F$ there exists $\langle K_F, \D_F \rangle \in \G \cap B_F$. Because $\G$ is a filter for every $\langle K, \D \rangle \in \G$ there exists $\langle L_F, \E_F \rangle \in \G$ such that $\langle L_F, \E_F \rangle \leq_{\PP} \langle K, \D \rangle$ and $\langle L_F, \E_F \rangle \leq_{\PP} \langle K_F, \D_F \rangle$. It follows that $K \setminus K_F \subseteq L_F \setminus K_F \subseteq \bigcap \D_F \subseteq F$. Thus $G \setminus K_F \subseteq F$ and $G \subseteq^{\ast} F$.
\end{proof}

\section{Main result}

We will use the fact that rapid ultrafilters are precisely those ultrafilters
which have nonempty interesection with every tall summable ideal. Thus in order
to construct an $\I_g$-ultrafilter which is not rapid, we want to construct an
$\I_g$-ultrafilter which has an empty intersection with another summable
ideal $\I_h$.

\begin{lemma} \label{Istep}
($MA_{\sigma-\hbox{\small centered}}$) Assume $\I_g$ and $\I_h$ are two tall summable ideals 
such that $\I_g \not \leq_K \I_h$.  Assume $\F$ is a filter base with $|\F| < \ccc$ 
such that $\F \cap \I_h = \emptyset$ and a function $f \in \omega^{\omega}$ is given.
Then there exists $G \subseteq \omega$ such that $f[G] \in \I_g$ and $G \cap F
\not\in \I_h$ for every $F \in \F$.
\end{lemma}

\begin{proof}
We may apply Lemma \ref{oneset} on $\F$ and $\I_h$. So there is a $H \not\in \I_h$ 
such that $|H \setminus F| < \omega$ for every $F \in \F$.

If $f[H] \in \I_g$ then put $G = H$.

If $f[H] \not\in \I_g$ we may apply Lemma \ref{choosingA} which provides $A
\subseteq f[H]$ such that $A \in \I_g$ and $f^{-1}[A] \cap H \not\in \I_h$. Put 
$G = f^{-1}[A]$.

$\bullet$ $f[G] = A \in \I_g$

$\bullet$ Since $G \cap H \not\in \I_h$ and $(G\cap H) \setminus F$ is finite
for every $F \in \F$ it follows that $G \cap F \not\in \I_h$ for every $F \in \F$.
\end{proof}

\begin{theorem} \label{IgnotIh}   % preparation for the main theorem
($MA_{\sigma-\hbox{\small centered}}$) For arbitrary tall summable ideals $\I_g$ and $\I_h$ such that $\I_g \not
\leq_K \I_h$ there is an $\I_g$-ultrafilter $\U$ with $\U \cap \I_h = \emptyset$.
\end{theorem}

\begin{proof}
Enumerate all functions in ${}^{\omega}\omega$ as $\{f_{\alpha}:\alpha <
\ccc\}$. By transfinite induction on $\alpha < \ccc$ we construct filter
bases $\F_{\alpha}$ such that the following conditions are satisfied:
\smallskip

(i) $\F_0$ is the Fr\'echet filter

(ii) $\F_{\alpha} \supseteq \F_{\beta}$ whenever $\alpha \geq \beta$

(iii) $\F_{\gamma} = \bigcup_{\alpha < \gamma} \F_{\alpha}$ for $\gamma$ limit

(iv) $(\forall \alpha)$ $|\F_{\alpha}| \leq |\alpha + 1| \cdot \omega$

(v) $(\forall \alpha)$ $\F_{\alpha} \cap \I_h = \emptyset$

(vi) $(\forall \alpha)$ $(\exists F \in \F_{\alpha+1})$ $f_{\alpha}[F] \in \I_g$
\smallskip

Conditions (i)--(iii) allow us to start the induction and keep it going.
Moreover (iii) ensures that (iv)--(vi) are satisfied at limit stages of the
construction, so it is necessary to verify conditions (iv)--(vi) only at
non-limit steps.

Induction step: Suppose we already know $\F_{\alpha}$.

Due to (iv) and (v) we may apply Lemma \ref{Istep} to $f_{\alpha}$ and
$\F_{\alpha}$. Let $\F_{\alpha+1}$ be the filter base generated by $\F_{\alpha}$
and $G$. The filter base $\F_{\alpha+1}$ satisfies (iv)--(vi).
\smallskip

Finally, let $\F = \bigcup_{\alpha < \ccc} \F_{\alpha}$. Because of condition (vi) 
every ultrafilter which extends $\F$ is an $\I_g$-ultrafilter. Because
of condition (v) $\F$ has empty intersection with $\I_h$ and thus can be
extended to an ultrafilter $\U$ with $\U \cap \I_h = \emptyset$. 
%It follows that $\U$ is not an $\I_h$-ultrafilter.
\end{proof}

\begin{proposition} \label{noMinimals}
For every tall summable ideal $\I_g$ there is a tall summable ideal $\I_h$ such that $\I_g \not\leq_K \I_h$.
\end{proposition}

\begin{proof}
Since $\I_g$ is a tall summable ideal we may fix a partition of $\omega$ into finite consecutive intervals
$I_n$, $n \in \omega$ such that

(i) $I_0 \neq \emptyset$

(ii) $|I_{n+1}| \geq n |\bigcup_{j \leq n} I_j|$ for every $n \in \omega$

(iii) for every $n > 0$ if $m \in I_n$ then $g(m) < \frac{1}{2^n}$
\medskip

\noindent
Now define $h: \omega \rightarrow (0,\infty)$ by
$$h(m) = \left\{\begin{array}{ll}
1 & \hbox{for } m \in I_0 \\
\frac{1}{n} & \hbox{for } m \in I_n \hbox{ with } n \geq 1
\end{array} \right.$$
It remains to verify that $\I_g \not\leq_K \I_h$. We will show that for every $f: \omega \rightarrow \omega$
there exists $A \in \I_g$ such that $f^{-1}[A] \not\in \I_h$.

Consider $f: \omega \rightarrow \omega$ arbitrary. For every $n \in \omega$ define
$$B_n = \{m \in I_n: f(m) < \min I_n\} \quad \qquad C_n = \{m \in I_n: f(m) \geq \min I_n\}$$

{\sl Case I. $A_0 = \{n \in \omega: |B_n| \geq |C_n|\}$ is infinite}
\smallskip

Since $B_n \cup C_n = I_n$ we have $|B_n| \geq \frac{1}{2} |I_n| \geq \frac{n}{2} |\bigcup_{j < n} I_j| = \frac{n}{2} (\min I_n - 1)$.
Thus for every $n \in A_0$ there exists $m_n \in f[B_n]$ such that $|f^{-1}(m_n) \cap B_n| \geq \frac{n}{2}$.

If $A = \{m_n: n \in A_0\}$ is finite then, of course $A \in \I_g$. Otherwise there exists an infinite set $A \subseteq \{m_n: n \in A_0\}$ such that $A \in \I_g$ because $\I_g$ is a tall ideal. In both cases $\tilde{A}_0 =\{n \in A_0: m_n \in A\}$ is infinite and $f^{-1}[A] \not\in \I_h$ because
$$\sum_{a \in f^{-1}[A]} h(a) \geq \sum_{n \in \tilde{A}_0} \sum_{a \in f^{-1}[A] \cap I_n} h(a) \geq \sum_{n \in \tilde{A}_0} |f^{-1}(m_n) \cap I_n| \cdot \frac{1}{n} \geq \sum_{n \in \tilde{A}_0} \frac{1}{2} = \infty$$

\medskip

{\sl Case II. $A_0 = \{n \in \omega: |B_n| \geq |C_n|\}$ is finite}
\smallskip

According to the assumption there is $n_0 \in \omega$ such that $|B_n| < |C_n|$ for every $n \geq n_0$.
Pick $m_n \in C_n$ for every $n \geq n_0$. Put $M = \{m_n: n \geq n_0\}$ and $A = f[M]$.
Since $m_n \in C_n$ one has $f(m_n) \geq \min I_n$ and therefore $g(f(m_n)) \leq \frac{1}{2^n}$.
It is easy to see that $A \in \I_g$ because
$$\sum_{a \in A} g(a) \leq \sum_{n \geq n_0} g(f(m_n)) \leq \sum_{n \geq n_0} \frac{1}{2^n} = \frac{1}{2^{n_0-1}}.$$
It remains to verify that $f^{-1}[A] \not\in \I_h$. To see this notice that
$$\sum _{a \in f^{-1}[A]} h(a) \geq \sum_{a \in M} h(a) = \sum_{n \geq n_0} h(m_n) = \sum_{n \geq n_0} \frac{1}{n} = \infty.$$
\end{proof}

\begin{theorem} \label{Ignotrapid}   % main theorem
%(MA${}_{\hbox{\small ctble}}$)
($MA_{\sigma-\hbox{\small centered}}$) For an arbitrary tall summable ideal $\I_g$ there is an
$\I_g$-ultrafilter which is not rapid.
\end{theorem}

\begin{proof}
This is an immediate consequence of Theorem \ref{IgnotIh}, Proposition \ref{noMinimals} and the characterization
of rapid ultrafilters in Theorem \ref{MainRapid}.
\end{proof}

\section{One possible generalization and its limits}

Once we have proved Theorem \ref{Ignotrapid}, which so to speak reverses
Corollary \ref{rapidnotI}, we may ask whether it is possible that an
ultrafilter is an $\I_g$-ultrafilter for ``many" tall summable ideals
simultaneously and still not a rapid ultrafilter. Certainly, ``many" cannot
mean all tall summable ideals, because of Theorem \ref{herrapid}.
We will show that in fact $\ddd$ many may be too much, but less than $\bbb$ is not.
					
\begin{proposition}  \label{atmostD} % d many summable ideals may characterize her. rapid ultfs
There exists a family $\mathcal{D}$ of tall summable ideals such that
$|\mathcal{D}| = \ddd$ and an ultrafilter $\U \in \omega^{\ast}$ is rapid if and
only if it has a nonempty intersection with every tall summable ideal in
$\mathcal{D}$.
\end{proposition}

\begin{proof}
Let us first construct the family $\mathcal{D}$:
Assume $\mathcal{F} \subseteq {}^{\omega}\omega$ is a dominating family and
$|\mathcal{F}| = \ddd$. Without loss of generality we may assume that all
functions in $\mathcal{F}$ are strictly increasing and $f(j+1) \geq f(j)+j+1$
for every $j \in \omega$.
For every $f \in \mathcal{F}$ define $g_f: \omega \rightarrow (0,+\infty)$ by
$$g_f(m)=\left\{
	\begin{array}{l @{\quad} l}
	1 & \hbox{if } m < f(0) \\
	\frac{1}{j+1} & \hbox{if } m \in [f(j),f(j+1))
	\end{array}
	\right.
$$
Let $\mathcal{D}=\{\I_{g_f}: f \in \mathcal{F}\}$.

Now, one implication is clear since every rapid ultrafilter has a nonempty
intersection with all tall summable ideals, in particular it has a nonempty
intersection with every ideal from $\mathcal{D}$.

It remains to verify that if an ultrafilter has nonempty intersection with
every ideal in $\mathcal{D}$, then it has nonempty intersection with all tall
summable ideals and therefore is rapid.
To this end, assume $\I_g$ is an arbitrary tall summable ideal. One can define
a strictly increasing function $f_g$ such that for every $j \in \omega$:
\begin{itemize}
\item $f_g(j+1) \geq f_g(j)+j+1$
\item if $m \geq f_g(j)$ then $g(m) \leq \frac{1}{2^j}$
\end{itemize}

Remember that family $\mathcal{F}$ was dominating. Hence there exists $f \in
\mathcal{F}$ and $k_0 \in \omega$ such that $f(k) \geq f_g(k)$ for every $k
\geq k_0$. For a every $n \geq f(k_0)$ there exists a unique $j \geq k_0$ such
that $n \in [f(j),f(j+1))$. Since $n \geq f(j) \geq f_g(j)$ we get $g(n) \leq
\frac{1}{2^j} \leq \frac{1}{j+1} = g_f(n)$. From $g \leq^{\ast} g_f$ follows
that $\I_{g_f} \subseteq \I_g$. Thus every ultrafilter $\U \in \omega^{\ast}$
which has a nonempty intersection with all ideals from $\mathcal{D}$ has a
nonempty intersection with $\I_g$ and since $\I_g$ was arbitrary, $\U$ is a
rapid ultrafilter,
\end{proof}
						
\begin{proposition} \label{Bcentr}   % summable ideals are b-centered
If $\mathcal{D}$ is a family of tall summable ideals and $|\mathcal{D}| < \bbb$
then there exists a tall summable ideal $\I_g$ such that $\I_g \subseteq \I_h$
for every $\I_h \in \mathcal{D}$.
\end{proposition}
		
\begin{proof}
For every $\I_h \in \mathcal{D}$ define a strictly increasing function $f_h \in
{}^{\omega}\omega$ such that whenever $m \geq f_h(j)$ then $h(m) \leq
\frac{1}{2^j}$.

According to the assumptions, the family of functions $\mathcal{F}=\{f_h: \I_h
\in \mathcal{D}\}$ is bounded, so there exists $f \in {}^{\omega}\omega$
such that $f_h \leq^{\ast} f$ for every $f_h \in \mathcal{F}$.
We may assume that $f$ is strictly increasing.
Define $g: \omega \rightarrow (0,+\infty)$ by
$$g(m)=\left\{
	\begin{array}{l @{\quad} l}
	1 & \hbox{if } m < f(0) \\
	\frac{1}{j+1} & \hbox{if } m \in [f(j),f(j+1))
	\end{array}
	\right.
$$

For a given function $f_h \in \mathcal{F}$ there exists $k_h \in \omega$ such
that $f_h(k) \leq f(k)$ for every $k \geq k_h$.  For every $n \geq f(k_h)$
there is exactly one $j \geq k_h$ such that $n \in [f(j),f(j+1))$. Since $n
\geq f(j) \geq f_h(j)$ we get $h(n) \leq \frac{1}{2^j} \leq \frac{1}{j+1} =
g(n)$. From $h \leq^{\ast} g$ follows that $\I_g \subseteq \I_h$.
\end{proof}
						
\begin{corollary} % for less than b summable ideals there is an Ig-ultf for all of them at once which is not rapid
($MA_{\sigma-\hbox{\small centered}}$) If $\mathcal{D}$ is a family of tall summable ideals and $|\mathcal{D}| < \ccc$
then there exists an ultrafilter $\U \in \omega^{\ast}$ such that $\U$ is an
$\I$-ultrafilter for every $\I \in \mathcal{D}$, but $\U$ is not a rapid
ultrafilter.
\end{corollary}

\begin{proof}
Combine Theorem \ref{Ignotrapid} and Proposition \ref{Bcentr} and the fact that $\bbb = \ccc$ under $MA_{\sigma-\hbox{\small centered}}$.
\end{proof}

\section{Open questions}

Let $\mathcal{D}$ be a family of tall summable ideals.

\begin{question}
What is the minimal size of the family $\mathcal{D}$ such that rapid ultrafilters
can be characterized as those ultrafilters on the natural numbers which have
a nonempty intersection with all ideals in the family $\mathcal{D}$?
\end{question}

Due to Proposition \ref{atmostD} the size of such a family is at most $\ddd$. But is $\ddd$ really the minimum?

%\begin{question}
%Is it true that whenever the cardinality of $\mathcal{D}$ is less than $\ddd$
%then there exist an ultrafilter on the natural numbers which is an $\I_g$-ultrafilter
%for every $\I_g \in \mathcal{D}$, but not a rapid ultrafilter?
%\end{question}


\begin{thebibliography}{9}

\bibitem
{B} J. Baumgartner, Ultrafilters on $\omega$, {\it J.
Symbolic Logic} {\bf 60}, no.~2, 624--639, 1995.

\bibitem
{FTh} J. Fla\v skov\'a, Ultrafilters and small sets, {\it Ph.D. thesis,\/} Charles University, Prague 2006.

\bibitem
{FALC} J. Fla\v skov\'a, $\I$-ultrafilters and summable ideals, in: {\it Proceedings of the 10th Asian Logic Conference\/} (Kobe 2008), 113 -- 123, World Scientific, Singapore, 2010.

\bibitem
{M-A} D. Meza Alc\'{a}ntara, Ideals and filters on countable sets. {\it Ph.D. thesis.\/} UNAM M\'{e}xico, 2009.

\bibitem
{Mi} A. W. Miller, There are no $Q$-points in Laver's model for the Borel
conjecture, {\it Proc. Amer. Math. Soc.} {\bf 78}, no.~1, 103--106, 1980.

\bibitem
{V} P. Vojt\'{a}\v{s}, On $\omega^{\ast}$ and absolutely divergent series, {\it Topology Proceedings} {\bf 19}, 335 -- 348, 1994.

\end{thebibliography}
\end{document}